\documentclass{amsart}

\usepackage{amscd}
\usepackage{textcomp}
\usepackage{amsfonts}
\usepackage{amsthm}
\usepackage{amsmath}
\usepackage{txfonts}
\usepackage{graphicx}
\usepackage{wrapfig}
\usepackage{subcaption}
\usepackage{ascmac}
\usepackage{bm}
\usepackage{indentfirst}
\usepackage{comment}
\usepackage{amssymb}
\usepackage{textcomp}
\usepackage{xcolor}

\usepackage{tikz}
\usetikzlibrary{calc}

\newtheorem{thm}{Theorem}[section]
\newtheorem{rem}[thm]{Remark}
\newtheorem{defi}[thm]{Definition}
\newtheorem{ex}[thm]{Example}

\newtheorem{lem}[thm]{Lemma}
\newtheorem{prop}[thm]{Proposition}

\newtheorem{con}[thm]{Conjecture}

\def\NZQ{\Bbb}

\def\RR{{\NZQ R}}
\def\CC{{\NZQ C}}

\def\PP{{\NZQ P}}

\def\SS{{\NZQ S}}

\def\ml{\mathcal{C}}
\def\ml1{\mathcal{C}^1}
\def\mlb1{\mathcal{C}_{b}^{1}}

\def\frk{\frak}

\def\Phi{{\frk n}}

\def\rank{{\rm rank}}

\def\A{{\mathcal A}}
\def\B{{\mathcal B}}

\def\L{{\mathcal L}}

\newcommand{\R}[0]{\mathbb{R}}
\newcommand{\C}[0]{\mathbb{C}}

\newcommand{\Cal}[1]{\mathcal{#1}}

\newcommand{\codim}[0]{\mbox{codim }}

\usepackage{color}
\usepackage{ifthen}

\newboolean{usecolor}
\setboolean{usecolor}{true}

\ifthenelse{\boolean{usecolor}}
{
  \definecolor{colore}{cmyk}{0,1,0.6,0}
  \definecolor{coloregen}{cmyk}{0.7,0,1,0}
  \definecolor{coloresimo}{cmyk}{1,0.6,0,0}
}
{
  \definecolor{colore}{cmyk}{0,0,0,1}
  \definecolor{coloregen}{cmyk}{0,0,0,1}
  \definecolor{coloresimo}{cmyk}{0,0,0,1}
}

\title{The generalized Sylvester's and orchard problems via discriminantal arrangement}
\author{P. Das}
\author{E. Palezzato}
 \author{S. Settepanella*}
\address{%
Department of Mathematics,
Hokkaido University, Japan.}
\email{pragnyadas555@gmail.com}
\email{palezzato@math.sci.hokudai.ac.jp}
\email{s.settepanella@math.sci.hokudai.ac.jp}

\thanks{}

\begin{document}

\maketitle


\begin{abstract}
In 1989 Manin and Schechtman defined the discriminantal arrangement $\B(n, k,\A)$ associated to a generic arrangement $\A$ of $n$ hyperplanes in a $k$-dimensional space. An equivalent notion was already introduced by Crapo in 1985 with the name of geometry of circuits. While both those papers were mainly focused on the case in which $\B(n, k, \A)$ has a constant combinatorics when $\A$ changes, it turns out that the case in which the combinatorics of $\B(n, k, \A)$ changes is quite interesting as it classifies special configurations of points in the $k$-dimensional space. In this paper we provide an example of this fact elucidating the connection between the well known generalized Sylvester's and orchard problems and the combinatorics of $\B(n, k, \A)$. In particular we point out how this connection could be helpful to address those old but still open problems.
\end{abstract}


\section{Introduction}
The study of special configurations of lines and points in the plane is a classic subject even if several problems have been solved only recently or are still open. Among them the generalized Sylvester's problem and orchard problem play a key role in our paper. \\
In 1893 Sylvester \cite{Syl} posed the problem ``\textit{prove that it is not possible to arrange any finite number of real points so that a right line through every two of them shall pass through a third, unless they all lie in the same right line}".
This problem was solved, among others, in 1944 by Gallai (see \cite{Gal}) and his solution is known nowadays as the Sylvester-Gallai Theorem. If we call a line passing through exactly two points \textit{ordinary}, a generalization of this problem is to count the minimum number of ordinary lines given $n$ points in $\RR^2$. 
Green and Tao in 2013, \cite{Green}, described the most advanced result about ordinary lines proving that if $n$ is sufficiently large, then there are at least $[\frac{n}{2}]$ ordinary lines. 
The Sylvester's problem can also be posed in its dual form:
``given $n$ lines in the real projective plane, there must be a point of intersection of just two lines'', and the point is called \textit{ordinary point}. 
The dual of Green and Tao's Theorem states that $n$ lines in projective plane intersect in at least $[\frac{n}{2}]$ ordinary points if $n$ is large enough.\\
The orchard problem was posed around 200 years ago and in its more general form asks which is the maximum number $t_3(n)$ of 3-point lines attainable with a configuration of $n$ points. Its dual form is equivalent to ask which is the maximum number of multiplicity $3$ intersections attainable with a configuration of $n$ lines. So far the most advanced known result is due to Burr, Gr\"{u}nbaum and Sloane (see \cite{Grum}) in 1974,where the authors provided a lower bound for $t_3(n)$.\\
In this paper we connect those two problems to the study of the combinatorics of the discriminantal arrangement associated to non very generic arrangements of $n$ lines in the affine and projective real plane. Introduced by Manin and Schechtman in 1989, the {\it  discriminantal arrangement} $\B(n, k, \A)$ is an arrangement of hyperplanes, constructed from a generic arrangement $\A$, generalizing classical braid arrangements (cf. \cite{man} p.209). The combinatorics of $\B(n, k, \A)$ is known to be constant when $\A$ varies in an open Zarisky set $\mathcal{Z}$ in the space of all generic arrangements. In \cite{athana}, Athanasiadis provided a description of this combinatorics proving a conjecture by Bayer and Brandt,  \cite{BB}, who called $\A$ \textit{very generic} if $\A \in \mathcal{Z}$ and \textit{non very generic} otherwise. We will keep those names throughout all the paper.\\
Few years before the work of Manin and Schechtman, Crapo \cite{crapo} introduced a notion equivalent to the discriminantal arrangement one 
starting from realizable matroids and called them \textit{geometry of circuits}. In his paper Crapo provided an example of an arrangement $\A$ of 6 lines in the real plane such that $\A\notin \mathcal{Z}$, pointing out how the directions at infinity of the lines in such arrangement give rise to a special configuration of $6$ points in the projective line.\\  
In 2016, Libgober and third author \cite{sette} gave a sufficient geometric condition for an arrangement $\A$ to be non very generic. 
In particular, in the case $k = 3$, their result shows that multiplicity $3$ codimension $2$ intersections of hyperplanes in $\B(n, 3, \A)$ appear (i.e. $\A \notin \mathcal{Z}$) if and only if collinearity conditions for points at infinity of lines, intersections of certain planes in $\A$, are satisfied. Thus confirming what already pointed out by Crapo, i.e. that combinatorics of the discriminantal arrangement $\B(n, k, \A)$ for $\A \notin \mathcal{Z}$ is strongly related to the problem of special configurations of points in the space. 
Two subsequent papers by Sawada, Yamagata and the third author (see \cite{SoSuSi}, \cite{SuSiSo}) went in the same direction connecting the combinatorics of $\B(6, 3, \A)$, for $\A \notin \mathcal{Z}$, to the Pappus's hexagon Theorem and the Hesse configuration in the complex case. \\
In this paper we firstly connect the dual orchard problem to the combinatorics of $\B(n, 2, \A)$, with $\A$ non very generic arrangement of $n$ lines in the real plane, then we provide an example of how to build, from a non very generic arrangement of $6$ planes in the real space, two special configuration of $12$ lines in the projective real plane one with minimum number of ordinary points and one with maximum number of multiplicity $3$ intersections. Hence we state Theorem \ref{thm:main} which provide a purely combinatorial way to build such examples and, finally, we generalize Theorem \ref{thm:main} into the Conjecture \ref{con} which should provide a way to inductively build arrangements of $n$ lines in real projective plane with low number of ordinary points or high number of multiplicity $3$ points starting from of a non very generic arrangement of $m < < n$ planes in the real space.\\
The content of the paper is as follows. In Section \ref{pre} we recall the definition of the discriminantal arrangement and few basic known results. In Section \ref{sec:motivex} we analyze Crapo's example, providing several examples of different combinatorics of $\B(6,2,\A)$, when $\A$ varies among all generic arrangements in $\RR^2$. 
In Section \ref{sec:Orch} we relate the dual orchard problem to the combinatorics of $\B(n, 2, \A)$, with $\A$ a non very generic arrangement in the real plane. Finally in Section \ref{sec:main} we state the main Theorem of the paper together with our Conjecture.


\section{Preliminaries}\label{pre}

\subsection{Discriminantal arrangement}\label{discarr}

Let $H^0_i, i=1,...,n$ be a generic arrangement $\A$ in $\CC^k, k<n$ i.e. 
a collection of hyperplanes such that $\codim \bigcap_{i \in K,
 \mid K\mid=p} H_i^0=p$ for any $0 \leq p \leq k$. The space of parallel translates $\SS(H_1^0,...,H_n^0)$ (or simply $\SS$ when the 
dependence on $H_i^0$ is clear or not essential)
is the space of $n$-tuples $H_1,...,H_n$ such that either $H_i \cap H_i^0=\emptyset$ or $H_i=H_i^0$ for any $i=1,...,n$.
One can identify $\SS$ with the $n$-dimensional affine space $\CC^n$ in such a way that $(H^0_1,...,H^0_n)$ corresponds to the origin.  In particular, an ordering of hyperplanes in $\A$ determines the coordinate system in $\SS$ (see \cite{sette}).\\
 We will use the compactification of $\CC^k$ viewing it as $\PP^k(\CC)\setminus H_{\infty}$ endowed with collection of hyperplanes $\bar H^0_i$ which are projecive closures of affine hyperplanes $H^0_i$.\\
Given a generic arrangement $\A$ in $\CC^k$ formed by hyperplanes $H^0_i, i=1,...,n$ {\it the trace at infinity}, denoted by $\A_{\infty}$, is the arrangement formed by hyperplanes $H_{\infty,i}=\bar H^0_i\cap H_{\infty}$. The trace $\mathcal{A}_{\infty}$ of an arrangement $\mathcal{A}$ determines the space of parallel translates $\mathbb{S}$ (as a subspace in the space of $n$-tuples of hyperplanes in $\mathbb{P}^k$).\\
Consider the closed subset of $\mathbb{S}$ formed by those collections which fail to form a generic arrangement. This subset of $\mathbb{S}$ is a union of hyperplanes $D_L \subset \mathbb{S}$ (see \cite{man}). Each hyperplane $D_L$ corresponds to a subset $L = \{ i_1, \dots, i_{k+1} \} \subset$  [$n$] $\coloneqq \{ 1, \dots, n \}$ and it consists of $n$-tuples of translates of hyperplanes $H_1^0, \dots, H_n^0$ in which translates of $H_{i_1}^0, \dots, H_{i_{k+1}}^0$ fail to form a general position arrangement. The arrangement $\B(n, k, \A)$ of hyperplanes $D_L$ is called \textit{discriminantal arrangement} and has been introduced by Manin and Schechtman in \cite{man}.  Notice that $\B(n, k, \A)$ depends on the trace at infinity $\mathcal{A}_{\infty}$ hence it is sometimes more properly denoted by $\B(n, k,\mathcal{A}_{\infty})$.\\
Let $\alpha_i = (a_{i1}, \dots , a_{ik}$) be  the normal vectors of hyperplanes $H^0_i$, $1 \leq i \leq n$, in the generic arrangement $\mathcal{A}$ in $\C^k$. Normal here is intended with respect to the usual dot product $$(a_1, \ldots, a_k)\cdot (v_1,\ldots, v_k)=\sum_i a_iv_i \quad .$$
Then the normal vectors to the hyperplanes $D_L$, $L = \{i_1 < \dots < i_{k+1} \} \subset$ [$n$] in $ \mathbb{S} \simeq \mathbb{C}^n$ are nonzero vectors of the form 
\begin{equation}\label{eq:normvec}
\alpha_L = \sum^{k+1}_{j=1} (-1)^i \det (\alpha_{i_1}, \dots, \hat{\alpha_{i_j}}, \dots, \alpha_{i_{k+1}})e_{i_j} \quad ,
\end{equation} 
where $\{e_j\}_{1\leq j \leq n}$ is the standard basis of $\mathbb{C}^n$ (cf. \cite{BB}). 

\subsection{Intersection lattice of the discriminantal arrangement}
It is well known (see, among others \cite{CrapoCTS},\cite{man}) that there exists an open Zarisky set $\mathcal{Z}$ in the space of generic arrangements of $n$ hyperplanes in $\RR^k (\CC^k)$, such that the intersection lattice of the discriminantal arrangement $\mathcal{B}(n,k,\A)$ is independent from the choice of the arrangement $\A \in  \mathcal{Z}$. Bayer and Brandt in \cite{BB} call the arrangements $\A \in  \mathcal{Z}$ \textit{very generic} and the one which are not in $\mathcal{Z}$, \textit{non very generic}. 
We will use their terminology in the rest of this paper. The name very generic is motivated by the fact that, in this case the number of intersections in the intersection lattice $\mathcal{L}(\mathcal{B}(n,k,\A))$ is the largest possible between all the discriminantal arrangements $B(n,k,\A')$, when $\A'$ ranges between all generic arrangements of $n$ hyperplanes in $\RR^k (\CC^k)$.\\
In \cite{CrapoCTS} Crapo proved that the intersection lattice of $\mathcal{B}(n,k,\A), \A \in  \mathcal{Z}$ is isomorphic to the Dilworth completion of the $k$-times lower-truncated Boolean algebra $B_n$ ( see Theorem 2. page 149 ). A more precise description of this lattice in the real case is due to Athanasiadis who proved in \cite{athana} a conjecture by Bayer and Brandt which stated that the intersection lattice of the discriminantal arrangement is isomorphic to the collection of all sets $\{S_1, \ldots, S_m\}$, where $S_i$ are subsets of $[n]$, each of cardinality at least $k+1$, such that
\begin{equation}\label{eq:vgcon}
\mid \bigcup_{i \in I} S_i \mid > k + \sum_{i \in I}(\mid S_i \mid - k) \mbox{ for all } I \subset [m]=\{1,\ldots,m\}, \mid I \mid \geq 2 \quad .
\end{equation}
The isomorphism is the natural one which associates to the set $S_i$ the space $D_{S_i}=\bigcap_{L \subset S_i, \mid L \mid=k+1} D_L$ intersection of all hyperplanes in the discriminantal arrangement indexed in $S_i$. In particular if the condition in equation (\ref{eq:vgcon}) is satisfied, this implies that the subspaces $D_{S_i}, i=1,\ldots, m$ intersect transversally or, equivalently, that $$\rank \bigcap_{i=1}^m D_{S_i}=\sum_{i=1}^m ( \mid S_i \mid -k )$$ being  $\rank~D_{S_i}=\mid S_i \mid -k$ (as proved in Corollary 3.6 in \cite{athana}).\\
Notice that while there are two different descriptions of the intersection lattice in the very generic case, very few is known about it in the non very generic case. 


\section{Motivating example: 6 lines in $\RR^2$}\label{sec:motivex}

In this section we study the combinatorics of the discriminantal arrangement $\mathcal{B}(6,2,\A)$ with $\A$ generic arrangement of $6$ lines in $\RR^2$. The arrangement $\mathcal{B}(6,2,\A)$ is a central arrangement in $\R^6$ essential in dimension $4$, i.e. its essentialization is an arrangement in $\R^4$. It has $\binom{6}{3}$ $=20 $ hyperplanes and in \cite{sette} authors show that its combinatorics in rank $2$ is constant equal to the one given by Athanasiadis for any choice of $\A$, that is it contains $\binom{6}{4}$ $=15$ intersections of multiplicity $4$ and the remaining intersections of multiplicity $2$. Hence in order to completely describe the combinatorics of $\mathcal{B}(6,2,\A)$ we only need to study the intersections in rank $3$.\\
In particular when $\A \in \Cal Z$ there are exactly $6$ intersections of multiplicity $5$ and the remaining intersections are either transversal of multiplicity $3$, i.e. $3$ hyperplanes intersecting transversally, or multiplicity $5$, i.e. an hyperplane intersecting transversally the multiplicity $4$ intersections in rank $2$. When $\A \notin \Cal Z$, then intersections of multiplicity $4$ can appear. The first to study this case has been Crapo in \cite{CrapoCTS}. Following his notation (see Section~7 of \cite{crapo}) we will call \textit{quadrilateral set} a configuration of six lines in the plane that intersect in exactly four triple points. By definition of discriminantal hyperplanes, if a generic arrangement $\A$ admits a translate which is a quadrilateral set, then $\A$ is non very generic. 


\begin{ex}\label{ex:2points}
Consider the quadrilateral set $\A^t$ shown in Figure~\ref{fig:triangles} translated of a generic arrangement $\A$.  Lines $\ell_1,\dots,\ell_6$ in $\A^t$ have normal vectors $$\alpha_1=(-2,2),\alpha_2= (-3, 4 ),\alpha_3=(0, 6 ), \alpha_4= (2,4), \alpha_5=(3,2) \mbox{ and }\alpha_6=(-1,2).$$  
In this case the $4$ triple intersection points in $\A^t$ correspond to the $4$ sets of $3$ indices $i$ of lines $\ell_i$ given by $I=\{\{1, 2, 3\}, \{1, 4, 6\}, \{2, 5, 6\}, \{3, 4, 5\}\}$.
This is equivalent to the fact that the hyperplanes $D_{L}, L \in I$ are $4$ hyperplanes in $\mathcal{B}(6,2,\A)$ with $$\A^t \in X=\cap_{L \in I} D_L \neq \emptyset \quad .$$ That is $X$ is an intersection of multiplicity $4$ in rank $3$. Indeed it cannot be an intersection in rank $4$ since $\mathcal{B}(6,2,\A)$ is central in rank $4$ and all arrangements belonging to the center of $\mathcal{B}(6,2,\A)$ are exactly the translated of $\A$ which are central arrangements.\\

\end{ex}
\begin{figure}[ht!]
    \centering
    \includegraphics[scale=0.18]{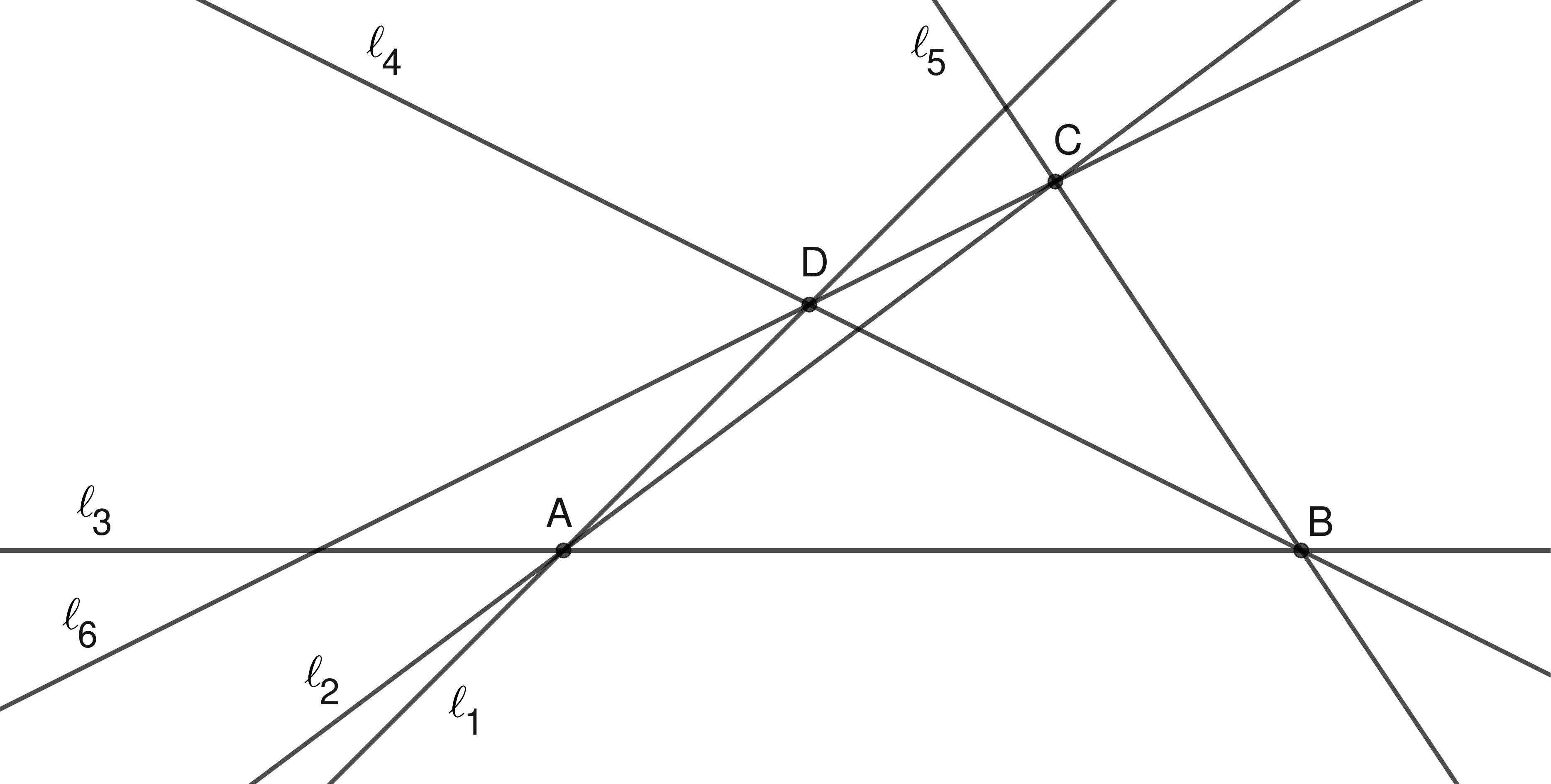}
    \caption{}
    \label{fig:triangles}
\end{figure}
\noindent
It is not difficult to verify that the arrangement $\A$ described in Example \ref{ex:2points} admits another translation $\A^{t'}$ which is a quadrilateral set with the $4$ triple points intersections of the lines indexed in  $I = \{\{1, 2, 6\}, \{1, 3, 4\}, \{2, 3, 5\}, \{4, 5, 6\}\}$. Indeed this follows from the general statement that quadrilateral sets always appear in couple (see \cite[Vol. I, p.62]{Baker}).\\
This implies that intersections in rank $3$ of multiplicity $4$ in $\mathcal{B}(6,2,\A)$ are always in even number. How many of those intersections can we have?


\begin{figure}[ht!]
    \centering
    \includegraphics[scale=0.18]{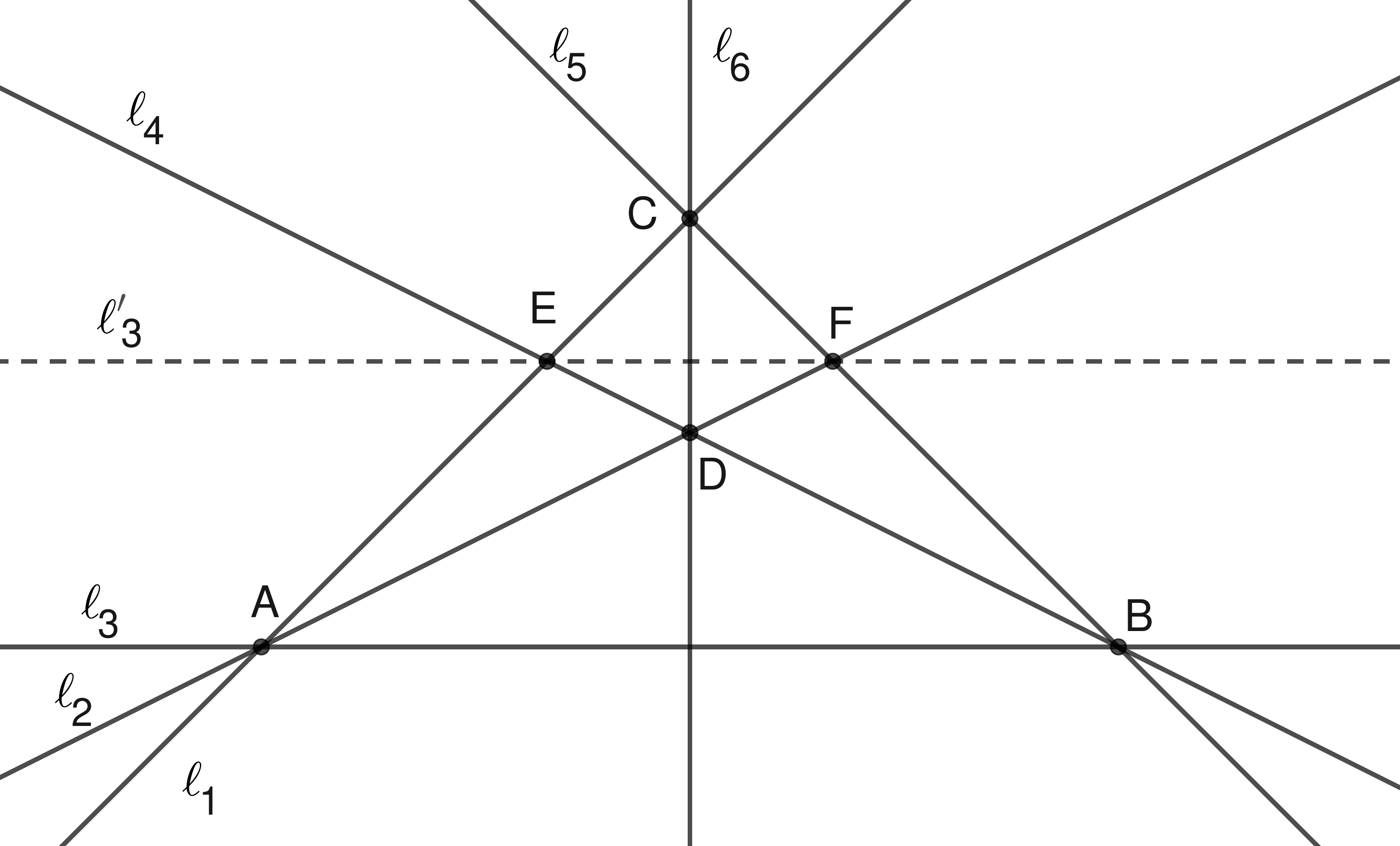}
    \caption{}
    \label{fig:IsoscelesTriangle}
\end{figure}
\begin{ex} \label{ex:4points}
Consider the quadrilateral set $\A^t$ shown in Figure~\ref{fig:IsoscelesTriangle} translated of a generic arrangement $\A$.  Lines $\ell_1,\dots,\ell_6$ in $\A^t$ have normal vectors $$\alpha_1=(-2,2),\alpha_2= (-2, 4 ),\alpha_3=(0, 6 ), \alpha_4= (2,4), \alpha_5=(2,2) \mbox{ and }\alpha_6=(1,0).$$  
In this case $\A$ admits $4$ different translated that are quadrilateral sets having triple points which are intersections of lines indexed in $I_1=\{\{1, 2, 3\}, \{1, 5, 6\}, \{2, 4, 6\}, \{3, 4, 5\}\}$ ( the one depicted in Figure~\ref{fig:IsoscelesTriangle}),$I_2=\{\{1, 2, 6\}, \{1, 3, 5\}, \{2, 3, 4\}, \{4, 5, 6\}\}$,$I_3=\{\{1, 3, 4\}, \{1, 5, 6\},$ $ \{2, 3, 5\}, \{2, 4, 6\}\}$ and $I_4=\{\{1, 3, 5\}, \{1, 4, 6\}, \{2, 3, 4\}, \{2, 5, 6\}\}$. 
\end{ex}
\noindent
Figure~\ref{fig:triangles} and Figure~\ref{fig:IsoscelesTriangle} show two examples with respectively one quadrilateral set and its second configuration, and two different quadrilateral sets and their second configurations.




 
\begin{figure}[ht!]
    \centering
    \includegraphics[scale=0.18]{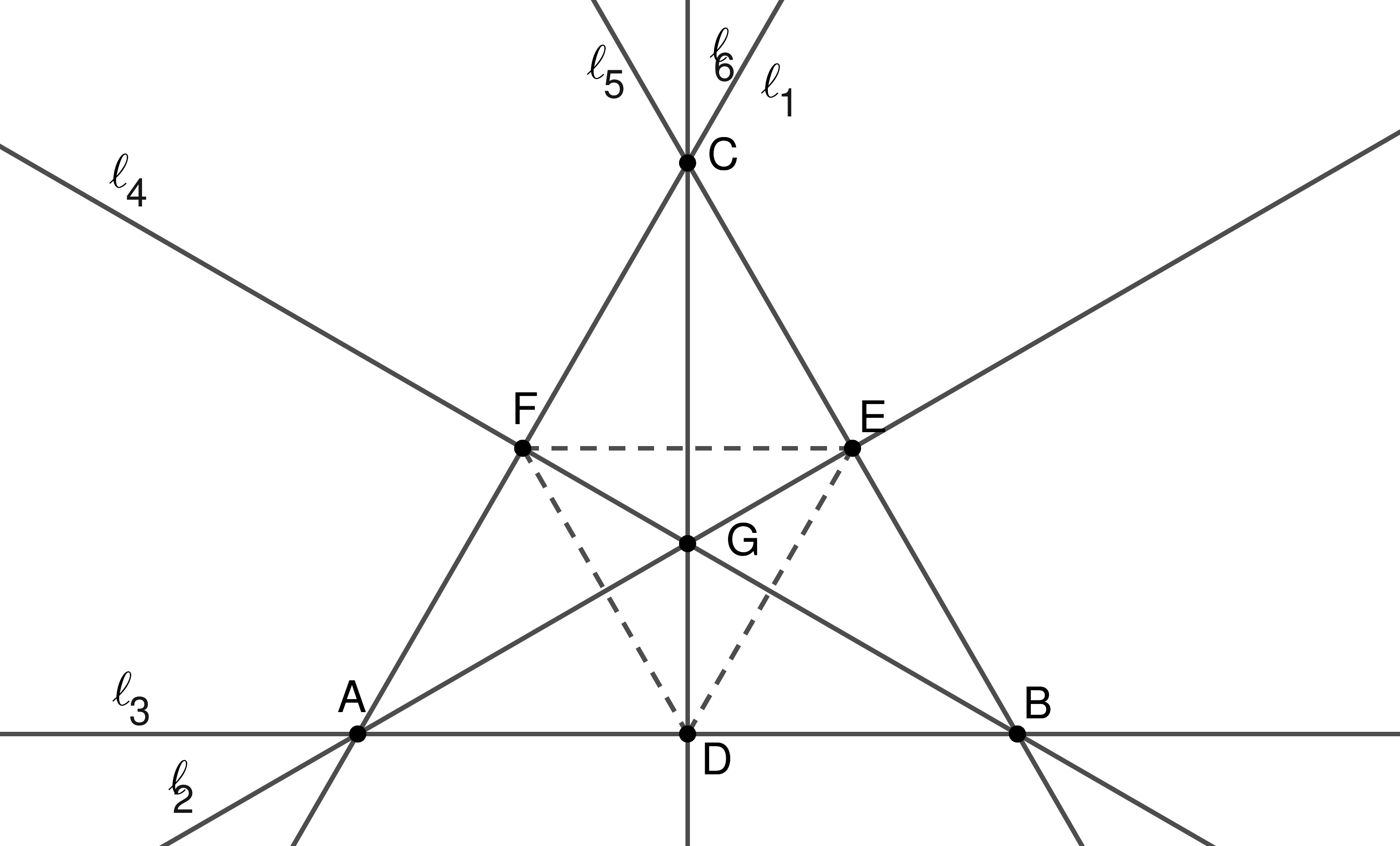}
    \caption{Equilateral triangle}
    \label{Figure:EquilateralTriangle}
\end{figure}

\begin{ex}
Consider the quadrilateral set $\A^t$ shown in Figure~\ref{Figure:EquilateralTriangle} translated of a generic arrangement $\A$.  Lines $\ell_1,\dots,\ell_6$ in $\A^t$ have normal vectors $$\alpha_1=(-1, \sqrt{3}),\alpha_2= (1, \sqrt{3} ),\alpha_3=(1, 0 ), \alpha_4= (-\sqrt{3},1), \alpha_5=(\sqrt{3},1) \mbox{ and }\alpha_6=(0,1).$$  
In this case $\A$ admits $8$ different translated that are quadrilateral sets having triple points intersections of lines indexed in $I_1=\{\{1, 2, 3\}, \{1, 5, 6\}, \{2, 4, 6\}, \{3, 4, 5\}\}$ ( the one depicted in Figure~\ref{Figure:EquilateralTriangle} ), $I_2=\{\{1, 2, 3\}, \{1, 4, 5\}, \{2, 4, 6\}, \{3, 5, 6\}\}$, $I_3=\{\{1, 2, 4\}, \{1, 3, 5\}, \{2, 3, 6\}, \{4, 5, 6\}\}$,
$I_4=\{\{1, 2, 5\}, \{1, 3, 6\}, \{2, 4, 6\}, \{3, 4, 5\}\}$, $I_5=\{\{1, 2, 6\}, \{1, 3, 5\}, \{2, 3, 4\}, \{4, 5, 6\}\}$, \\ $I_6=\{\{1, 2, 6\}, \{1, 4, 5\}, \{2, 4, 5\}, \{3, 4, 6\}\}$, $I_7=\{\{1, 3, 4\}, \{1, 5, 6\}, \{2, 3, 5\}, \{2, 4, 6\}\}$  \\ and $I_8=\{\{1, 3, 5\}, \{1, 4, 6\}, \{2, 3, 4\}, \{2, 5, 6\}\}$.
\end{ex}
\noindent
In order to complete the classification of all the possible combinatorics of $\mathcal{B}(6,2,\A)$, it is left to show if there exist real arrangements $\A$ such that $\mathcal{B}(6,2,\A)$ admits $m$ intersections of multiplicity $4$ in rank $3$ with $m=6$ and $m>8$. While we conjecture that the answer is yes for the case $m=6$ and no for $m>8$, this problem is left open.

\section{The orchard problem and the combinatorics of $\mathcal{B}(n,2,\A)$.}\label{sec:Orch}
The examples of Section \ref{sec:motivex} suggest that the orchard problem, i.e. the study of maximum number of intersections of multiplicity $3$ can be rewritten in terms of combinatorics of the discriminantal arrangement $\mathcal{B}(n,2,\A)$. In this section we introduce the Notation and the Definitions needed to formulate the orchard problem in terms of combinatorics of the discriminantal arrangement.

\subsection{Non very generic intersections}

Let $\A$ be a generic arrangement of $n$ lines in $\R^2$ ( equivalently $\C^2$) and denote by $\L_r(\mathcal{B}(n,2,\A))$ the rank $r$ elements in the intersection lattice of $\mathcal{B}(n,2,\A)$. In this case hyperplanes $D_{L} \in \mathcal{B}(n,2,\A)$ are indexed by subsets $L \subset [n]$, $\mid L \mid =3$. By combinatorics of very generic discriminantal arrangement we get that $r$ hyperplanes $D_{L_i}, i=1, \ldots , r$ either intersect transversally or there exists a partition $\{S_1,...,S_m\}$ of $\{1,\ldots ,r\}, |S_i|\geq3$ which satisfies Athanasiadis condition (in equation (\ref{eq:vgcon})) such that $\cap_{i=1}^r D_{L_i}=\cap_{i=1}^m D_{S_i}$ . Notice that in this case $D_{S_i}$ intersects transversally and $\rank({\cap_{i=1}^r D_{L_i}})=\sum_{i=1}^m ( \mid S_i\mid -2 )$ (see corollary 3.6 in \cite{athana}). On the other hand in the non very generic case we can have $r$ hyperplanes which do not intersect transversally even if for any subset $I \subset [r], \mid I \mid \geq 2, \cap_{i \in I}D_{L_i} \neq D_K \in \L(\mathcal{B}(n,2,\A)), K \subset [n], \mid K \mid > 3$. We recall the definition of simple intersection given in \cite{SY}.

\begin{defi}\label{def:tower} An element X $\in\L_r(\mathcal{B}(n,2,\A))$ is called a simple intersection if $X=\cap_{i=1}^m D_{L_i}$, $D_{L_i} \in \mathcal{B}(n,2,\A)$, $m\geq r$ and for every subset $I\subset [m],$ $\mid I \mid \geq 2,\cap_{i \in I}D_{L_i} \neq D_K \in \L(\mathcal{B}(n,2,\A)), K \subset [n], \mid K \mid > 3$. In particular if $m>r$ we call X a non very generic simple intersection. We call $m$ multiplicity of X.
\end{defi}

\begin{rem}\label{rem 4.2}
Notice that an element $\A'$ in a simple intersection X $\in\L_r(\mathcal{B}(n,2,\A))$ corresponds to a translate of the generic arrangement $\A$ containing only double and triple intersection points. Indeed an arrangement $\A'$ of n lines in $\R^2(\C^2)$ in which lines $l_i$'s intersect in points $P=\cap_{i \in K} l_i$,  $\mid K \mid > 3$ belongs to $D_K = \underset{\underset{\mid L \mid =3}{L \subset K}}{\cap D_L}$ by definition.\\
Viceversa if a translate $\A'$ of a generic arrangement $\A$ of $n$ lines in $\R^2(\C^2)$ contains lines intersecting only in double points and triple points $P_i, i=1, \ldots ,m$ of the form $P_i=\cap_{j \in L_i}l_j, L_i \subset [n], \mid L_i \mid=3$, then $\A'$ is an element in the simple intersection $ X=\cap_{i=1}^m D_{L_i} \in \L_r(\mathcal{B}(n,2,\A))$.
\end{rem}

\noindent Based on the above Remark \ref{rem 4.2} and description of the combinatorics of very generic arrangement provided by Athanasiadis, the following Proposition holds.

\begin{prop}\label{prop 4.3}
If a generic arrangement $\A$ of n lines in $\R^2(\C^2)$ is very generic then all simple intersection $ X=\cap_{i=1}^m D_{L_i}\in\L_r(\mathcal{B}(n,2,\A))$ have multiplicity $m$ equals to their rank $r$.

\end{prop}

 \begin{proof}
 If $\A$ is a very generic arrangement of n lines in $\R^2(\C^2)$ then any simple intersection $X=\cap_{i=1}^m D_{L_i}\in\L_r(\mathcal{B}(n,2,\A))$ satisfies the condition  $r=\rank(\cap_{i=1}^m D_{L_i})=\sum_{i=1}^m ( \mid L_i \mid-k )=\sum_{i=1}^m (3-2)=m$ (see corollary 3.6 in \cite{athana}), that is $r=m$. \\
 \end{proof}
 
 \noindent
 By definition of simple intersection $X\in\L_r(\mathcal{B}(n,2,\A))$, it follows that if the multiplicity $m$ is $m=r$ then $X$ is transversal intersection of exactly $m=r$ hyperplanes, that is $m=r$ is the minimum value of $m$ when the arrangement $\A$ varies among all generic arrangements of $n$ lines in $\R^2$ ( equivalently $\C^2$). One could ask which is the maximum value for $m$. Since Proposition \ref{prop 4.3} states that when the arrangement $\A$ is very generic then all simple intersections $X\in\L_r(\mathcal{B}(n,2,\A))$ have multiplicity $m=r$, then if $\A$ is an arrangement such that $\L_r(\mathcal{B}(n,2,\A))$ admits simple intersections of multiplicity $m>r$, then $\A$ has to be non very generic. While the fact that $m=r$ is minimum value follows trivially from the definition of simple intersection, to study the maximum value for $m$ is quite difficult but interesting problem. Indeed the following proposition holds.  
 
\begin{prop}\label{prop:equiv} A generic arrangement $\A$ of $n$ lines in $\R^2(\C^2)$ admits a translate $\A'$ which is an arrangement of $n$ lines in $\R^2(\C^2)$ with maximum number of multiplicity 3 intersection points if and only if it exists a simple intersection $X  \in \L_{n-3}(\B(n,2,\A))$ with maximum multiplicity $m$, i.e. for any generic arrangement $\overline{\A}$ of $n$ lines in $\R^2(\C^2)$ and simple intersections  $\overline{X}  \in \L_{n-3}(\B(n,2,\overline{\A}))$, multiplicity of $\overline{X} $ is smaller or equal than $m$.
\end{prop}

\begin{proof} By Remark \ref{rem 4.2} we know that an arrangement of n lines in $\R^2(\C^2)$ with only double and triple intersection points is an element in a simple intersection $X=\cap_{i=1}^m D_{L_i} \in\L_r(\B(n,2,\A))$ with multiplicity $m$ exactly equals to the number of triple intersection points. Hence the maximum number of triple points that an arrangement of n lines in $\R^2(\C^2)$ with only double and triple intersection can have is exactly equal to the maximum multiplicity that a simple intersection can have when both rank $r$ and $\A$ vary. Since  $n-3$ is the maximal rank in which non central arrangements appear, then the proof follows. 
\end{proof}
\noindent
Proposition \ref{prop:equiv} states the equivalence between the problem of studying simple non very generic intersections of $\B(n,2,\A)$ and the orchard problem.

\begin{ex}
Consider the arrangement $\A^t$ shown in Figure~\ref{fig:SevenLines} translated of a generic arrangement $\A$.  Lines $\ell_1,\dots,\ell_7$ in $\A^t$ have normal vectors $$\alpha_1=(-2, 2),\alpha_2= (-3,4),\alpha_3=(0,6 ), \alpha_4= (2,4), \alpha_5=(3,2), \alpha_6=(-1,2) \mbox{ and }\alpha_7=(-1.2,1.8).$$  Then $\A^t$ is an element in the not empty simple intersection $X=\cap_{L \in I} D_{L}$, \\$I= \{\{1, 2, 3\}, \{1, 4, 6\}, \{1, 5, 7\}, \{2, 4, 7\} \{2, 5, 6\}, \{3, 4, 5\}\}$. $X$ is a simple non very generic intersection in rank $7-3=4$ having multiplicity $6$, the maximum multiplicity attainable for such intersections when $\A \notin \Cal Z$.
\end{ex}
\begin{figure}[ht!]
    \centering
    \includegraphics[scale=0.18]{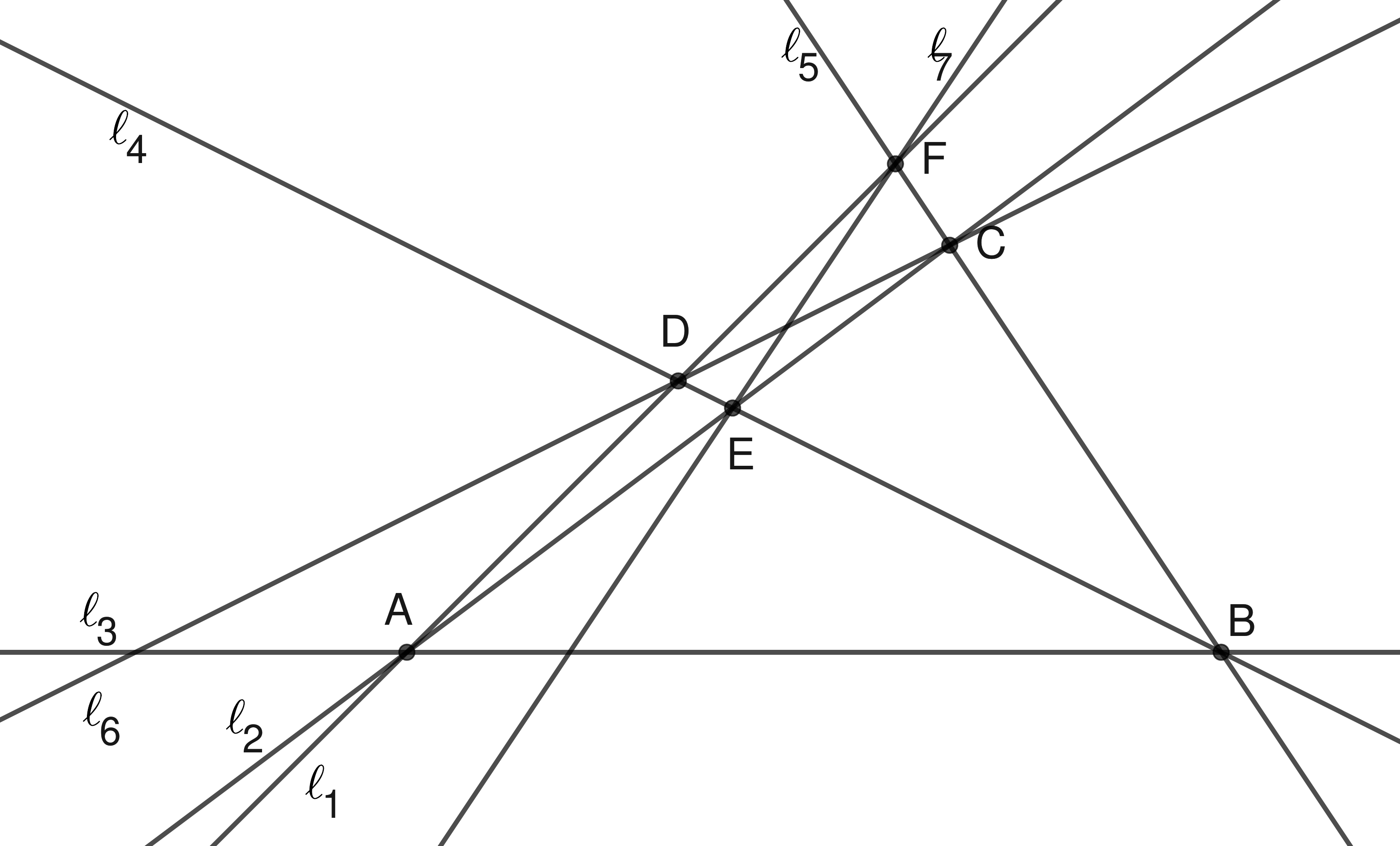}
    \caption{Seven lines arrangement}
    \label{fig:SevenLines}
\end{figure}


\section{Construction of 12 lines with 19 triple points from Pappus's configuration.}\label{sec:main}
In this section we will consider a generic arrangement $\A$ of planes in $\R^3$ and its trace at infinity $\mathcal{A}_\infty$, generic arrangement of lines in projective plane. \\
We begin by considering a generic arrangement $\A$ of $6$ planes in $\R^3$ whose trace at infinity $\mathcal{A}_\infty = \{ l_{i_1},l_{i_2},...,l_{i_6} \}$ of 6 lines in $\PP^2\R$ is represented in Figure \ref{Figure:pappus's}(A). 
\begin{figure}[ht!] 
  \centering
  \subfloat
  {\includegraphics[width=0.45\textwidth]{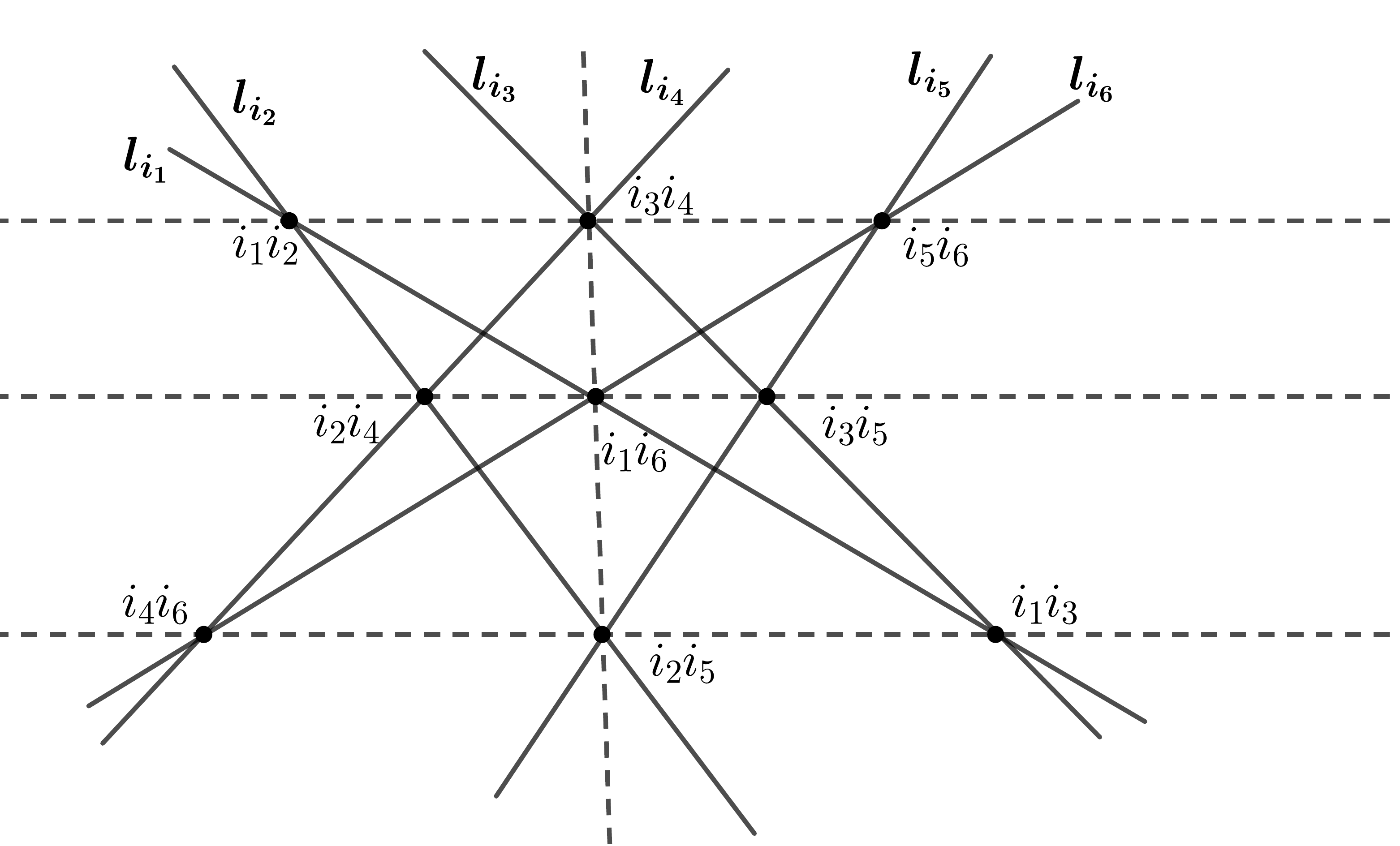}}
  \hfill
  \subfloat
  {\includegraphics[width=0.45\textwidth]{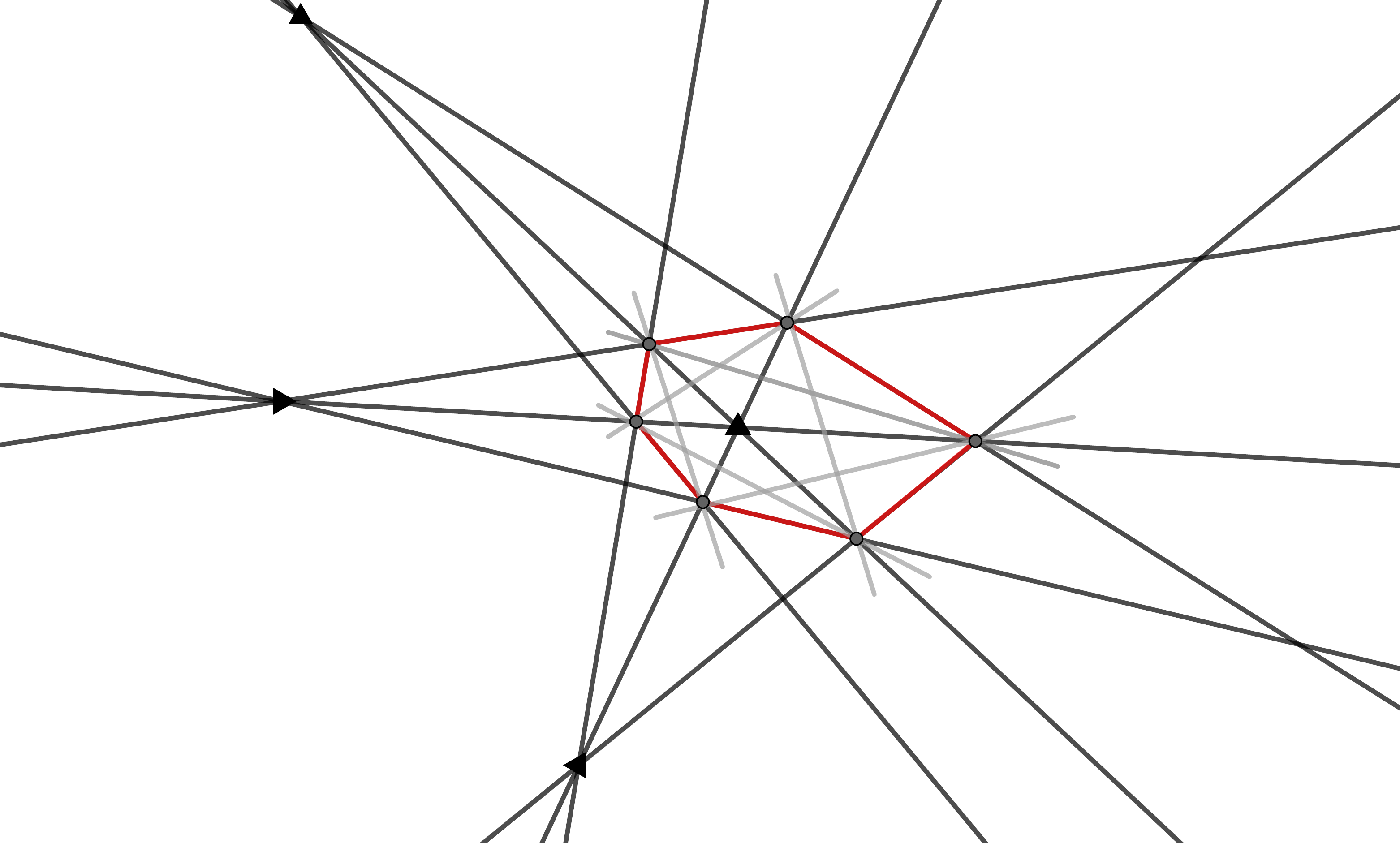}}
 \caption{ \small{(A)-left: Generic arrangement $\A_\infty$\label{Figure:pappus's} 
and (B)-right: the rank $2$ section of its discriminantal arrangement $\mathcal{B}(6,3,\A_\infty)$\label{Figure:Hexagon}}}
 \end{figure}
 \begin{figure}[ht!]
    \centering
    \includegraphics[scale=0.3]{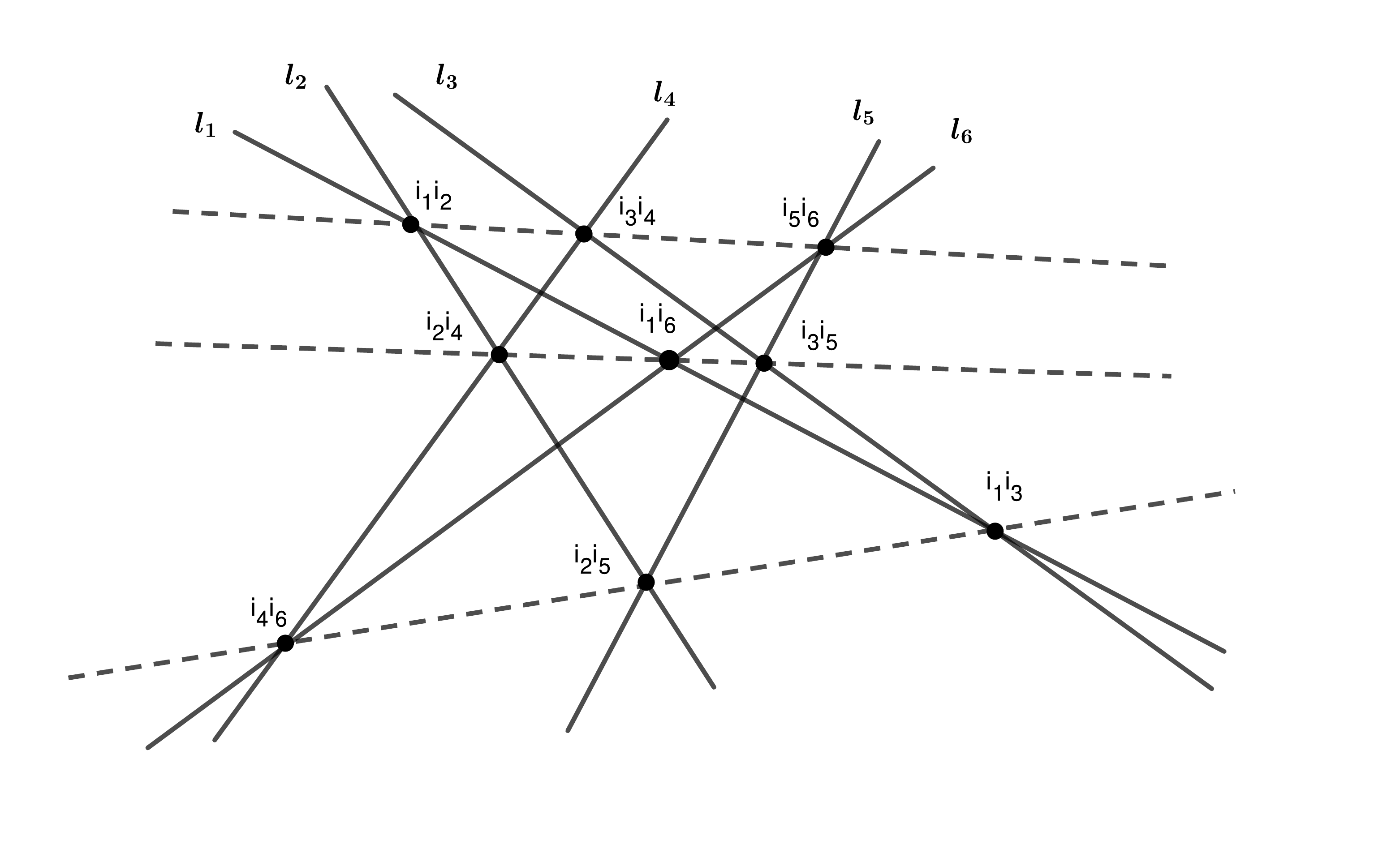}
    \caption{\small{Arrangement $\mathcal{P}_\infty$ corresponding to Pappus's configuration with 3 collinearity conditions.}}
    \label{fig:classical pappus}
\end{figure}
The lines in $\mathcal{A}_\infty$ intersect in $15$ double points satisfying $4$ collinearity conditions. For obvious reasons in the rest of this section we will refer to this configuration as Pappus's configuration with $4$ collinearities and denote it by $\mathcal{P}^c_\infty$, while $\mathcal{P}_\infty$ will denote the arrangement in Figure \ref{fig:classical pappus}, i.e. the Pappus's configuration with $3$ collinearities.\\
In \cite{sette} authors proved that collinearity conditions in $\mathcal{A_\infty}$ correspond to multiplicity 3 intersections in $\L_2(\B (n, 3, \mathcal{A_\infty}))$. In particular a generic planar section of $\mathcal{B}(6,3,\Cal P^c_\infty)$ is represented in Figure \ref{Figure:Hexagon}(B).\\
Our goal in this section is to show how the study of arrangements of projective lines with collinearities is connected to the study of arrangements of lines with maximum number of triple points and arrangements of lines with minimum number of double points. 
We will do this studying the case of $12$ lines in the plane builded along the following steps:
\begin{enumerate}
\item Consider the Pappus's configuration $\mathcal{P}^c_\infty$ (resp. $\mathcal{P}_\infty$ ) with $4$ (resp. $3$) collinearity conditions with the additional condition that the three lines corresponding to the three classical collinearities of Pappus's hexagon configuration are concurrent ( as depicted in Figure \ref{Figure:case6.4} where the three collinear lines are the parallel lines $l_1', l_3'$ and the line at infinity $l_2'$).
\item Add to this arrangement six more lines $\mathcal{P}'_\infty=\{l'_1,.....,l'_6\}$ as follows:
\begin{enumerate}
\item lines $l'_1, l'_2, l'_3$ are the three concurrent lines added in correspondence of the three Pappus's collinearities; 
\item lines $l'_4, l'_5, l'_6$ are added so that each one of them contains exactly two different double intersections of $\mathcal{P}^c_\infty$ (resp. $\mathcal{P}_\infty$ ) and that each double point is contained in only one line $l'_i, i=1,\ldots,6$.
\end{enumerate}
We call the so built arrangement $\mathcal{P}'_\infty$ a \textit{completion} of $\mathcal{P}^c_\infty$ (resp. $\mathcal{P}_\infty$ ).
\end{enumerate}
There are $3$ different completions $\mathcal{P}'_\infty$ of $\mathcal{P}^c_\infty$ and $\mathcal{P}_\infty$ . Indeed since lines in $\mathcal{P}^c_\infty$ ( resp. $\mathcal{P}_\infty$ ) intersect in $15$ double points, $9$ of which already belong to $l'_1, l'_2, l'_3$, the remaining $6$ double points\footnote{Notice that, by construction, any three of those $6$ points cannot be collinear, i.e. the construction in $(2)$(b) gives rise to three distinct lines. } can be partitioned in $3$ different ways as sets of couples.  
Two different arrangements $\mathcal{P}_\infty \cup \mathcal{P}'_\infty$ are depicted in  Figure \ref{Figure:case6.4} and Figure \ref{Figure:case6.3} while Figure \ref{Figure:case6.2} contains the arrangement $\mathcal{P}^c_\infty \cup \mathcal{P}'_\infty$ with the last possible alternative\footnote{In order to to simplify the figures we have chosen the line $l_2'$ to be the line at infinity.}.  With above notations, the following proposition holds.\\
\begin{prop}\label{prop:main} The arrangements of lines $\mathcal{P}_\infty\cup\mathcal{P}'_\infty$ are arrangements of $12$ lines in $\PP^2\R$ with maximum number of triple points and $\mathcal{P}^c_\infty\cup\mathcal{P}'_\infty$ is an arrangement of lines   with minimum number of double points.
\end{prop}

\begin{proof} It is well known that an arrangement of 12 lines in $\PP^2\R$ has at most $19$ triple points (see \cite{Grum} for details) or at minimum $6$ double points (see \cite{Green} for more details).  \\
The arrangement $\mathcal{P}^c_\infty\cup\mathcal{P}_\infty'$ represented in Figure \ref{Figure:case6.2} has exactly $6$ double intersections. Indeed an intersection of multiplicity 6 added to the 15 intersections of multiplicity $3$ leaves exactly $6$ intersections of multiplicity $2$ by the classical formula $$ \binom{s}{2}=\sum_{k\geq 2}t_k\binom{k}{2}
$$
where $s$ is number of mutually distinct lines and $t_k$ is the number of intersection points of multiplicity $k$ (see equation (1) in \cite{Dum}).\\
In the case depicted in Figure \ref{Figure:case6.4}, i.e.  $\mathcal{P}_\infty \cup \mathcal{P}_\infty'$, if we denote by $P_{ij}$ the double point given by the intersection $l_i \cap l_j$, the line $l'_5$ is chosen to be the line passing through double points $P_{23}$ and $P_{15}$, $l'_6$ the one passing through $P_{36}$ and $P_{45}$ and, finally, $l'_4$ is the one containing $P_{14}$ and $P_{26}$. 
This construction add $6$ triple points to the 10 ones obtained adding the lines $l'_1,l'_2$ and $l'_3$, for a total of 16 triple points. What is left to prove is that whenever we join lines in this way, if we denote by $P_{ij}
'=l'_i\cap l'_j$, the intersections $P_{56}', P_{45}'$ and $P_{46}'$ are triple points, that is, more precisely,  $P_{56}' \in l'_2, P_{45}' \in l'_1$ and $P_{46}' \in l'_3$. In order to prove this we remark that  the line $l'_5$ is parallel to $l'_6$ since, by simple geometric consideration, we have that the trapezium(or trapezoid) formed by $l'_5,l'_6,l_3,l_5$ has the parallel sides $\overline{P_{23}P_{36}}\text{ and }\overline{P_{15}P_{45}}$ of equal length making it a parallelogram. As a consequence $l'_2\cap l'_5\cap l'_6$ is not empty, that is $P_{56}' \in l'_2$. Same argument apply to the proof that $P_{45}' \in l'_1$ ( resp. $P_{46}' \in l'_3$ ) as soon as we consider the same configuration choosing the line $l'_1$ ( resp. $l'_3$ ) as the line at infinity as depicted in the left configuration of Figure \ref{Figure:case l1} (resp. right configuration of Figure \ref{Figure:case l3}). 
Same argument applies to the arrangement $\mathcal{P}_\infty \cup \mathcal{P}_\infty'$ depicted in Figure \ref{Figure:case6.3}.
\end{proof}
\begin{figure}[ht!]
    \centering
    \includegraphics[scale=0.3]{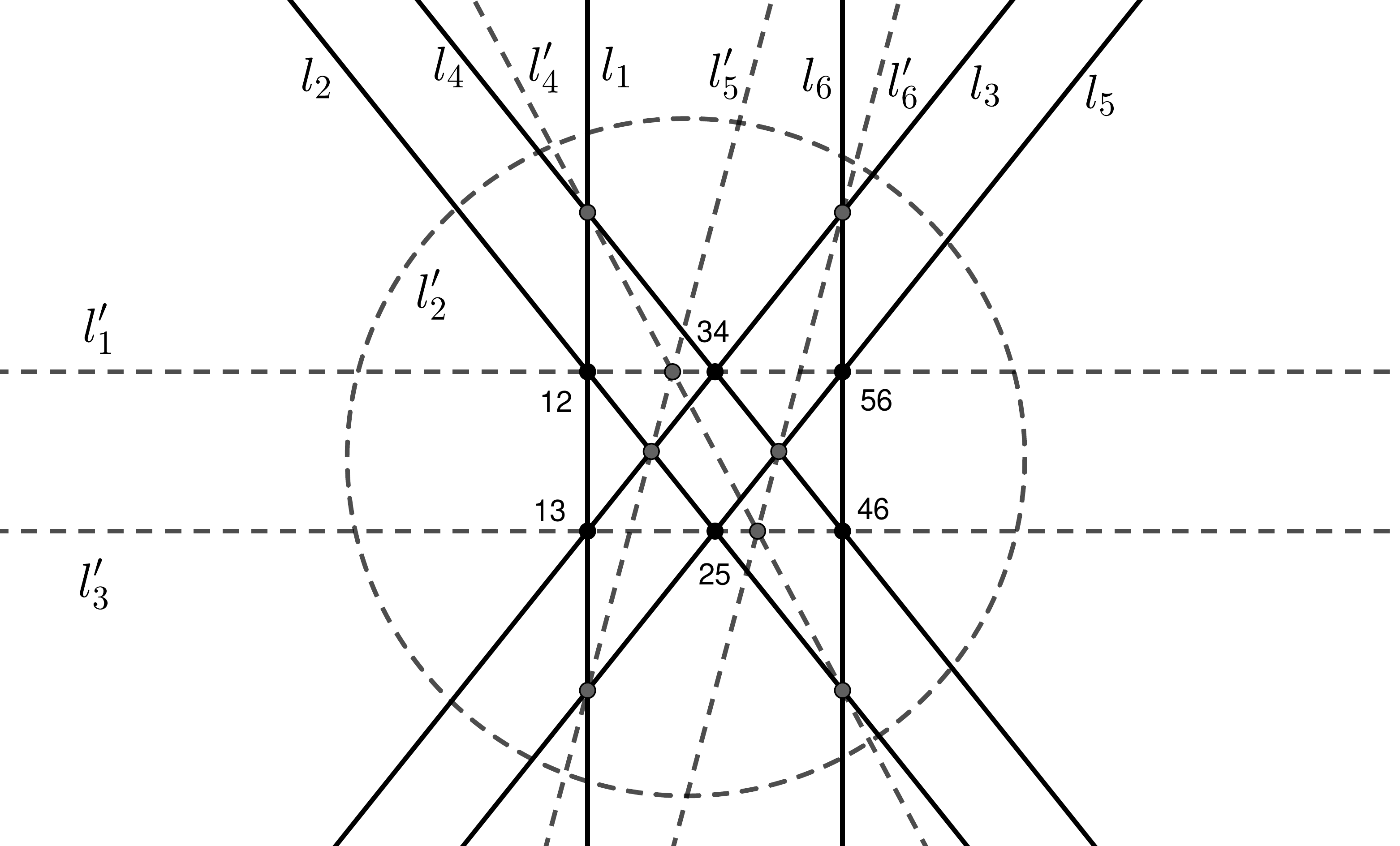}
    \caption{Arrangement  $\Cal P_\infty\cup\Cal P_\infty'$ with 19 3-points ($P_{ij}$ is written as $i j$). \label{Figure:case6.4}}
   
\end{figure}
\begin{figure}[ht!] 
  \centering
  \subfloat
  {\includegraphics[width=0.45\textwidth]{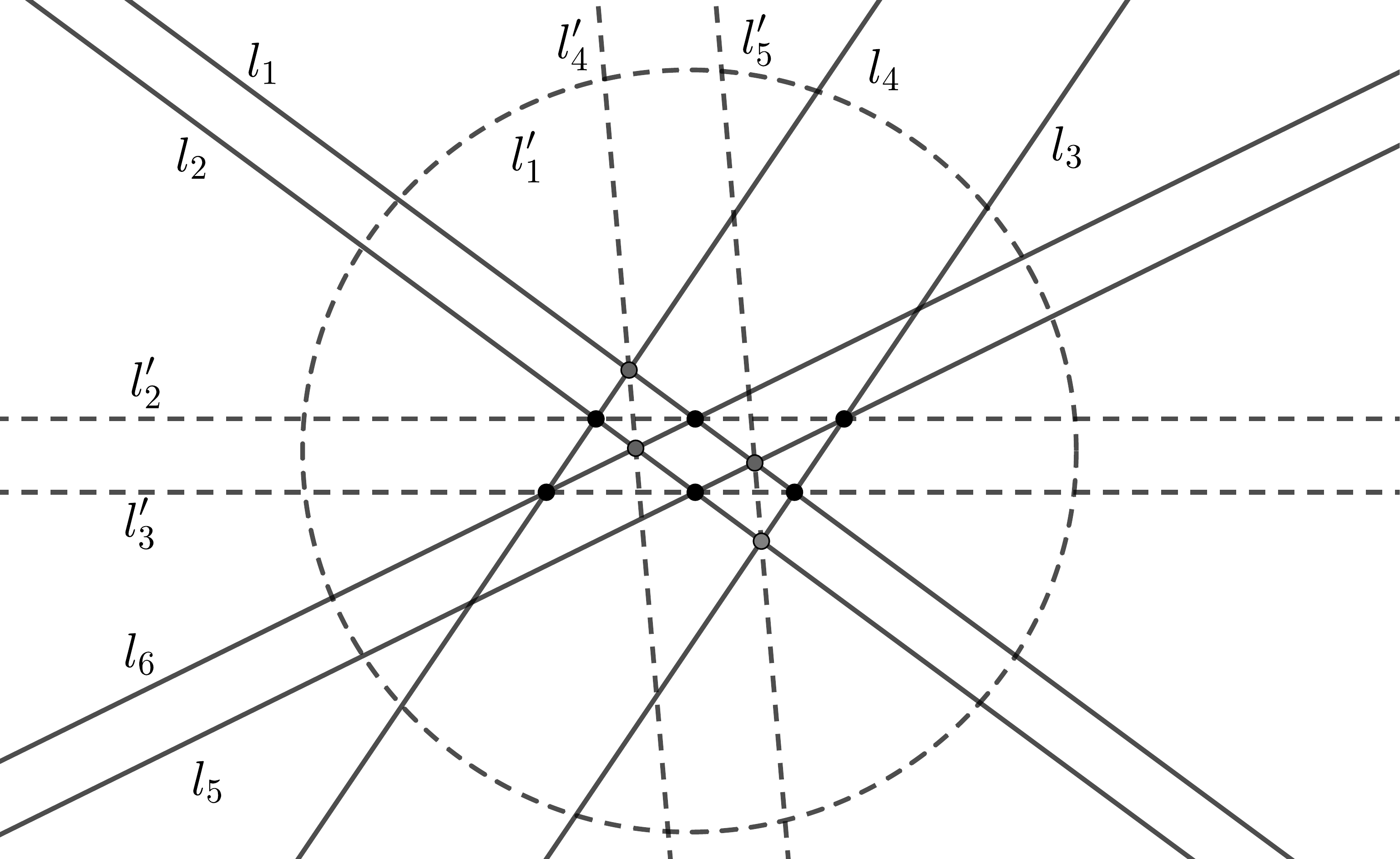}}
  \hfill
  \subfloat
  {\includegraphics[width=0.45\textwidth]{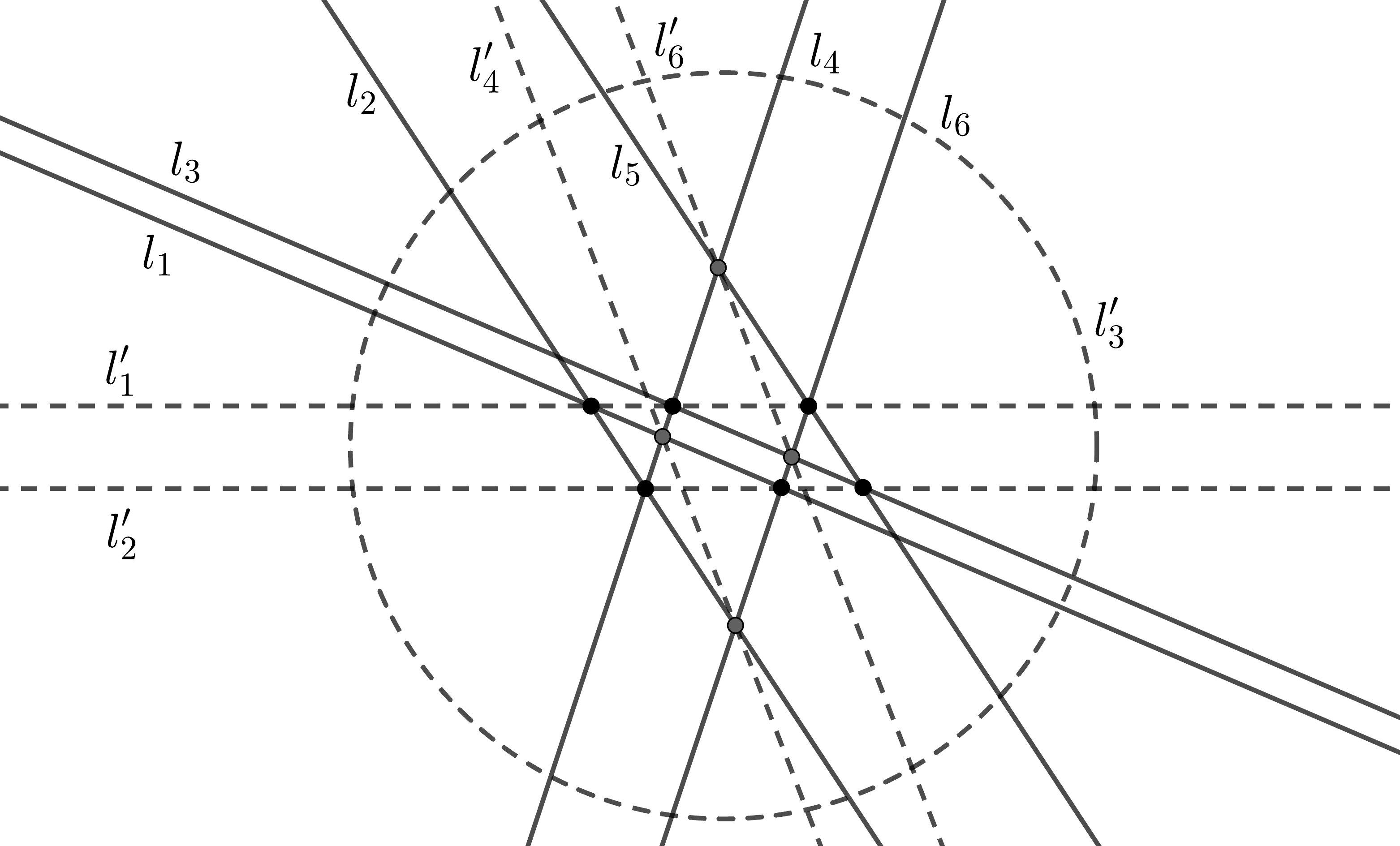}}
  \caption{on the left $\Cal P_\infty\cup\Cal P_\infty'$ with $l'_1$ as line at infinity \label{Figure:case l1}.On the right $l'_3$ is the line at infinity.\label{Figure:case l3}}.
  
\end{figure}
\begin{figure}[ht!] 
  \centering
  \includegraphics[scale=0.3]{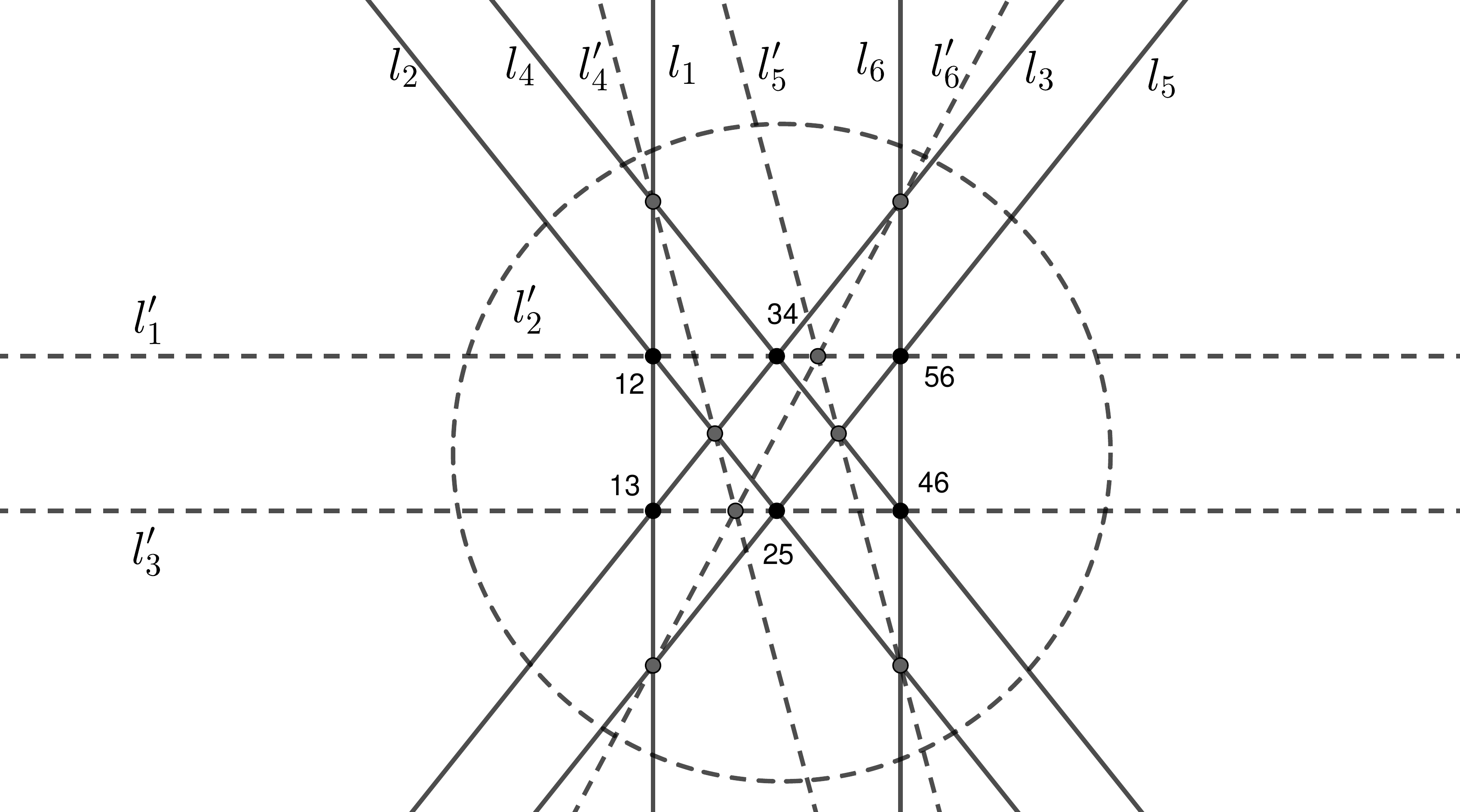}
  \caption{An arrangement  $\Cal P_\infty\cup\Cal P_\infty'$ with 19 3-points. $P_{ij}$ is written as $ij$. \label{Figure:case6.3}.}
\end{figure}

\begin{figure}[ht!] 
  \centering
  \includegraphics[scale=0.3]{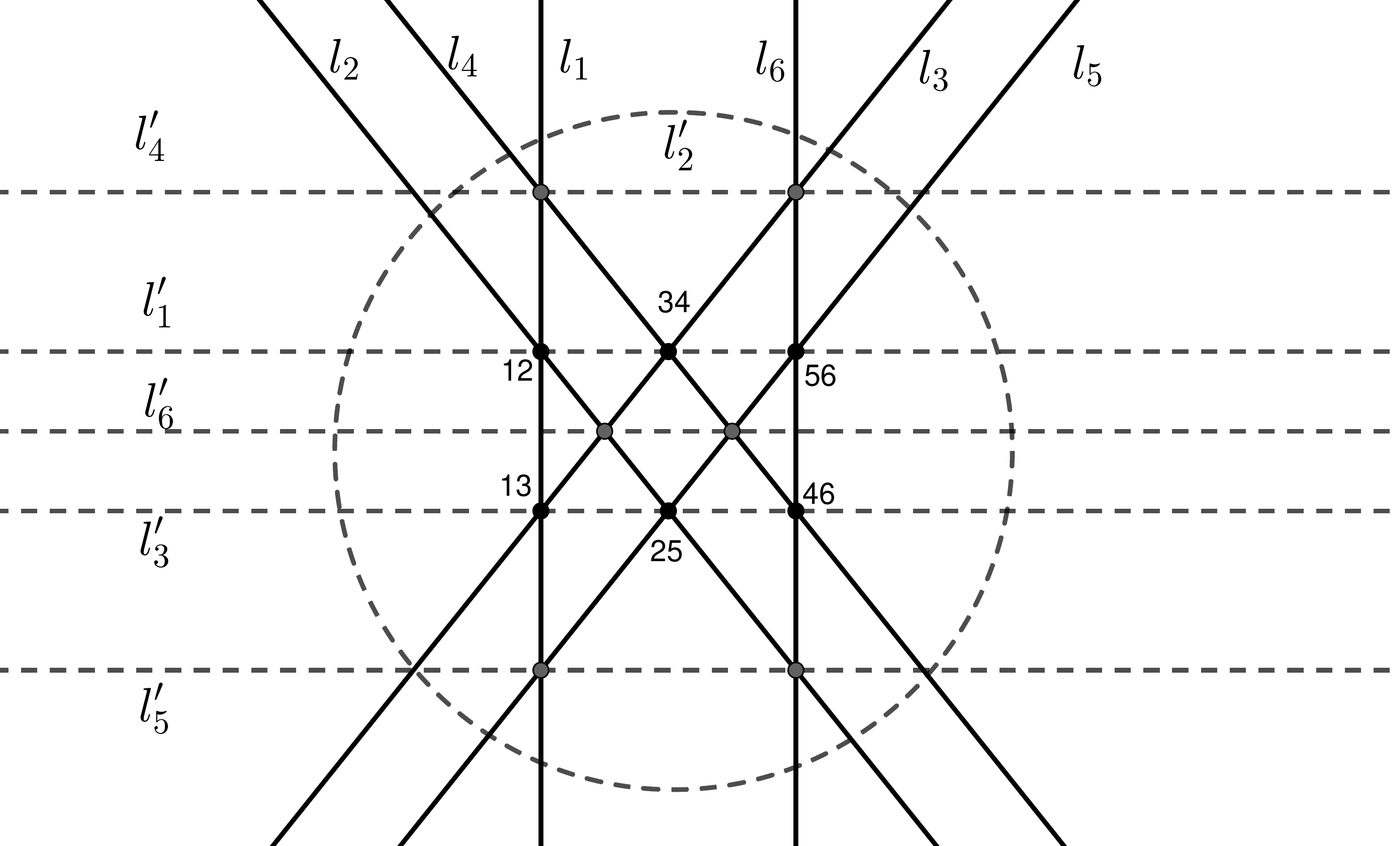}
  \caption{The arrangement  $\Cal P^c_\infty\cup\Cal P_\infty'$ with 6  2-points. $P_{ij}$ is written as $ij$. \label{Figure:case6.2}}
  
\end{figure}

\noindent
In the next subsection we will provide a way to build the completion $\mathcal{P}_\infty'$ that depends on the combinatorics of $\mathcal{B}(6,3,\Cal P^c_\infty)$ and $\mathcal{B}(6,3,\Cal P_\infty)$. This gives a purely combinatorial construction that can be generalized to a conjecture which we state in Subsection 5.3.

\subsection{Strong action on points $P\in\L_2(\mathcal{A}_\infty)$ and the $\sigma$-completion} Let's consider a generic arrangement $\A$ of $n$ planes in $\R^3$ with trace at infinity $\mathcal{A}_\infty = \{ l_1,l_2,...,l_n \}$ of n lines in $\PP^2\R$. Let $S \subseteq S_n$ be the subset of the symmetric group defined by
$$S=\{\sigma \in S_n \mid \sigma=\tau_1\tau_2....\tau_k, \quad \tau_i \mbox{ disjoint transpositions} \}  \quad .$$
The elements $\sigma\in S$ acts naturally on elements $P=l_m\cap l_n \in \L_2(\A_\infty)$ as $\sigma.P = l_{\sigma.m}\cap l_{\sigma.n}$.
We say that an element $\sigma \in S$ \textit{acts strongly} or that the action of $\sigma$ is \textit{strong} on $\A_\infty$ if it fixes non trivial collinearities in $\L_2(\A_\infty)$. That is if the points $\{P_1,P_2,\ldots,P_m\} , m\geq 3 $ in $\L_2(\mathcal{A}_\infty)$ are collinear and belonging to different lines, then $\sigma .P_i\in\{P_1,\ldots,P_m\}\text{ for any }i\in\{1,\dots,m\}.$ With the above notations, the following lemma holds.
\begin{lem}\label{lma5.1}
An element $\sigma \in S \subseteq S_6$ acts strongly on $\L_2(\mathcal{P}_\infty)$ if and only if $\sigma$ is a product of exactly three transpositions each one fixing a point in one of the three distinct collinearities. 
\end{lem}

\begin{proof}
Let the three classical collinearity conditions in the Pappus's configuration be $\mathcal{C}_1=\{P_{i_1i_2},P_{i_3i_4},P_{i_5i_6}\}$, $\mathcal{C}_2=\{P_{i_2i_4},P_{i_1i_6},P_{i_3i_5}\}$ and $\mathcal{C}_3=\{P_{i_1i_3},P_{i_2i_5},P_{i_4i_6}\}$ as depicted in Figure $\ref{fig:classical pappus}$.\\
Assume $\sigma$ fixes a line of $\mathcal{P}_\infty$. Without loss of generalities, we can assume that $\sigma$ fixes $l_{1_1}$. Then for any point $P_{i_1j}\in l_{i_1}$, we have $\sigma.P_{i_1j}=P_{i_1k}\in l_{i_1}$, for some $k\ne j$. As a consequence, $\sigma$ cannot act strongly on all non trivial collinearities as no collinearity contains two points in the same line in $\mathcal{P_\infty}$. It follows that if  $\sigma$ is strong then $\sigma$ cannot fix any line, i.e. it is a product of exactly three disjoint transpositions.\\
Moreover, if  $\sigma$ fixes all the three points in a non trivial collinearity then it is easy to check that $\sigma$ sends at least one point of one collinearity to a point in another that is $\sigma$ does not act strongly on $\L_2(\mathcal{P}_\infty)$.\\
Analogously, it is not difficult to check that if $\sigma $ permutes all the three points in a non trivial collinearity then $\sigma$ will be a product of non-disjoint transpositions. So $\sigma$ has to be the product of three disjoint transpositions that fixes exactly one point in each collinearity.
\end{proof}
\noindent
Let $\A$ be a generic arrangement of n planes in $\R^3(\C^3)$ with trace at infinity $\A_\infty$ and $\sigma\in S_n$ an element that acts strongly on $\A_\infty$. We call \textit{$\sigma$-completion} of $\A_\infty$, and denote it by $\A_\infty^\sigma$, the arrangement $\A_\infty^\sigma=\{l_1',\ldots,l_m'\}$ satisfying the following conditions:
\begin{itemize}
    \item $l_i'$ is the line $P\sigma.P$ where $P\in\L_2(\A_\infty)$.
    \item For any point $P\in\L_2(\A_\infty)$ there exists exactly one line $l_i'\in\A_\infty^\sigma$ such that $P\in l_i'$.
\end{itemize}
The Lemma \ref{lma5.1} implies that we have 3 different $\sigma$-completions $\Cal P_{\infty}^\sigma$ of $\mathcal{P}_\infty$.
As a consequence, for any given permutation $(i_1,i_2,\ldots,i_6),$  of the indices $\{1,\ldots, 6\}$, the set of $\sigma$ acting strongly on $\mathcal{P}_\infty$ is $$S_{\mathcal{P}_\infty}=\{\sigma_1=\tau_{i_1i_2}\tau_{i_3i_5}\tau_{i_4i_6},\sigma_2=\tau_{i_1i_3}\tau_{i_2i_4}\tau_{i_5i_6},\sigma_3=\tau_{i_1i_6}\tau_{i_2i_5}\tau_{i_3i_4}\} \quad .$$ 
The collinearities are $\mathcal{C}_1=\{P_{i_1i_2},P_{i_3i_4},P_{i_5i_6}\}$, $\mathcal{C}_2=\{P_{i_2i_4},P_{i_1i_6},P_{i_3i_5}\}$ and $\mathcal{C}_3=\{P_{i_1i_3},P_{i_2i_5},P_{i_4i_6}\}$ as depicted in Figure \ref{fig:classical pappus}. Remark that $\sigma_3$ is the only element which acts strongly on $\mathcal{P}^c_\infty$ too and the following proposition, equivalent to  Proposition \ref{prop:main} holds.\\
\begin{prop}\label{thm5.1}
If $\mathcal{P}_\infty$ and $\mathcal{P}^c_\infty$ satisfy the additional condition that three collinearities of the classical Pappus's configuration are concurrent then for $\sigma \in S_{\mathcal{P}_\infty}$ we have that
\begin{enumerate}
\item $\Cal P^c_\infty \cup\Cal P_\infty^\sigma$ is an arrangement with minimum number of 2-points if and only if $\sigma$ is the element that acts strongly on $\Cal P^c_\infty$,
\item $\Cal P_\infty \cup\Cal P_\infty^\sigma$ is an arrangement with maximum number of 3-points if $\sigma$ doesn't act strongly on $\Cal P^c_\infty$.
\end{enumerate}
\end{prop}

\subsection{Strong action on $\B(n,3,\A_\infty)$ and the line arrangement $\A^\sigma$}

The results in Subsection 5.1 can be easily expressed in terms of the combinatorics of the discriminantal arrangement $\B(n,k,\A_\infty)$. The symmetric group $S_n$ acts naturally on the hyperplanes $D_L\in \B(n,k,\A_\infty), L=\{s_1,\ldots,s_{k+1}\}\subset [n]$ as $$\sigma .D_L=D_{\sigma.L}\in\B(n,k,\A_\infty), \sigma.L=\{\sigma.s_1,\ldots,\sigma.s_{k+1}\} \quad . $$ Moreover since each collinearity condition in $\A_{\infty}$ corresponds to a point $P$ of multiplicity 3 in $\L_2(\B(n,3,\A_\infty))$ (see \cite{sette}), then $\sigma$ acts strongly on $\A_{\infty}$ if and only if it fixes the corresponding intersection $P\in \L_2(\B(n,3,\A_\infty))$\footnote{In details the collinearity condition $\mathcal{C}=\{P_{i_1i_2},P_{i_3i_4},P_{i_5i_6}\}$ in $\A_\infty$ corresponds to the 3-point $P=\cap_{i=1}^3D_{L_i}\in\L_2(\B(6,3,\A_\infty))$ with $L_1=\{i_3,i_4,i_5,i_6\},L_2=\{i_1,i_2,i_5,i_6\},L_3=\{i_1,i_2,i_3,i_4\}$.}. Accordingly we will say that $\sigma$ acts strongly on $\B(n,3,\A_\infty)$ if it fixes all triple intersections in $\L_2(\B(n,3,\A_\infty))$. Notice that this definition too has meaning only when the arrangement $\A$  is non very generic.\\
In order to rewrite Proposition \ref{thm5.1} only in terms of the combinatorics of the discriminantal arrangement we need two more steps. First we give the following definition.
\begin{defi}
Two intersections $P,P'\in\L_2(\B(n,3,\A_\infty))$ are \textit{independent} if and only if any hyperplane $D_L\in\B(n,3,\A_\infty)$ which contains $P$ does not contain the point $P'$. \\Given three pairwise independent intersections  $P_i\in\L_2(\B(n,3,\A_\infty)) ,i=1,2,3$ the intersection  $P\in\L_2(\B(n,3,\A_\infty))$ is \textit{purely dependent from $P_i$'s} if it is the intersection of three hyperplanes each one containing exactly one $P_i,i=1,2,3$ (as in Figure \ref{Pappushexagon}).
\end{defi}
\begin{figure}[ht!]
    \centering
    \includegraphics[scale=0.3]{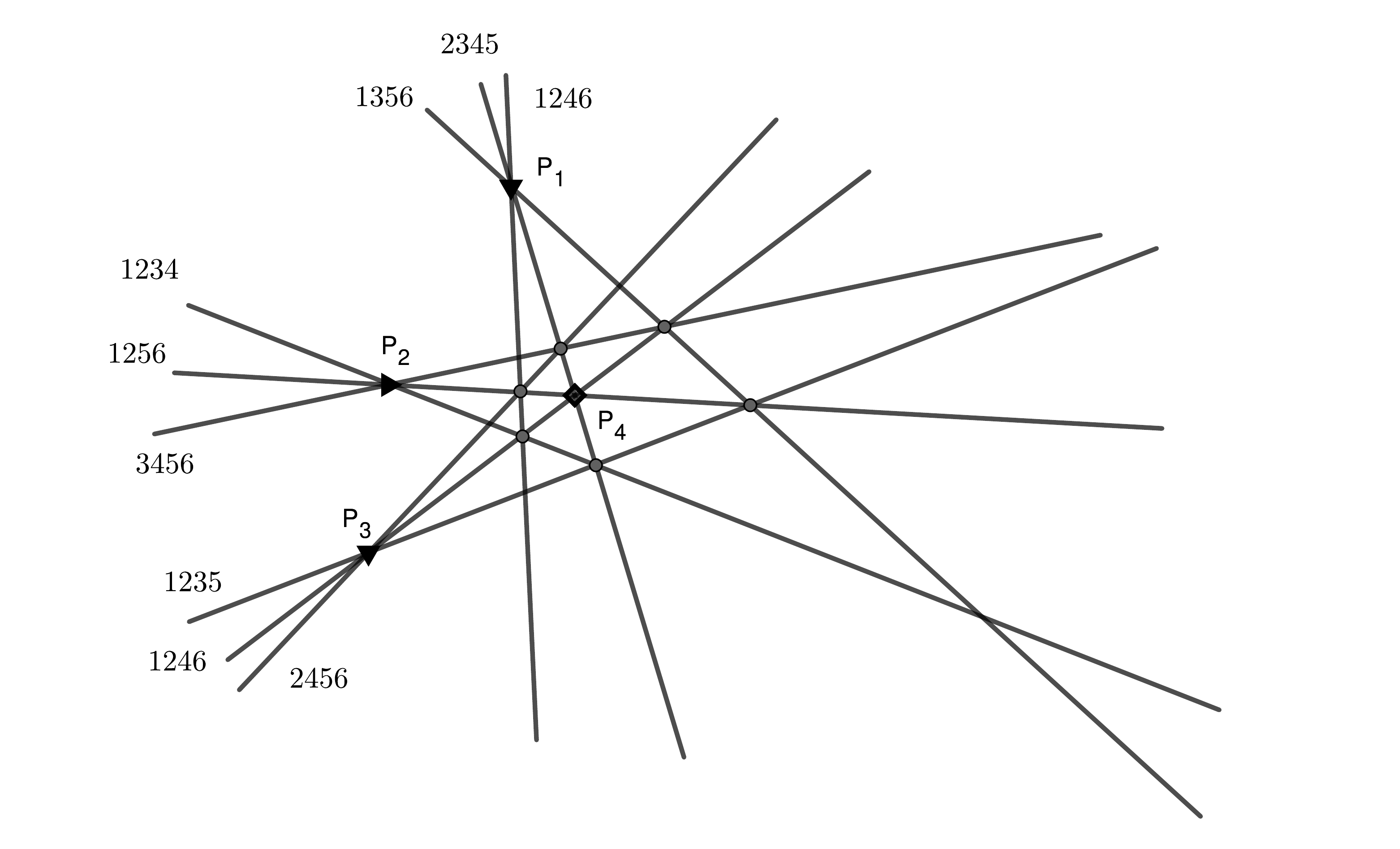}
    \caption{Discriminantal arrangement corresponding to the Pappus's configuration with 4 multiplicity 3 points}
    \label{Pappushexagon}
\end{figure}
\noindent
Remark that purely dependent intersections only exist when $\A$ is non very generic arrangement. Furthermore it is a simple remark that $\B(6,3,\mathcal{A}_\infty)$ contains at most $3$ independent intersections, realized when $\A_\infty=\Cal P_\infty$, and a purely dependent one as in $\Cal P^c_\infty$.   \\
Secondly let's denote by $\A^\sigma$ the arrangement in $\R^2(\C^2)$ whose projective closure is the $\sigma$-completion $\A_\infty^\sigma$ of $\A_\infty$. It is a simple remark that the $\sigma$ completion $\Cal P_\infty^\sigma$ at point $(2)$ of Proposition \ref{thm5.1} is an arrangement of $6$ lines with 4 triple points, that is the maximum possible ( see Figure \ref{Figure:case6.4}). Hence, for a suitable choice of the line at infinity, $\Cal P^\sigma$ is an arrangement in a simple non very generic intersection $X\in\L_3(\B(6,2,\A))$ of maximum multiplicity and the following theorem equivalent to Proposition  \ref{thm5.1} holds.\\
\begin{thm}\label{thm:main}
Let $\B(6,3,\mathcal{A}_\infty)$ be a discriminantal arrangement with maximum number of independent intersections in $\L_2(\B(6,3,\A_\infty))$ and $\sigma\in S$ an element that acts strongly on $\B(6,3,\mathcal{A}_\infty)$, then 
\begin{enumerate} 
\item The arrangement $\A_\infty\cup\A_\infty^\sigma$ is an arrangement with minimum number of intersections of multiplicity 2 if and only if $\L_2(\B(6,3,\A_\infty))$ contains a purely dependent intersection fixed by $\sigma$ and $\A_\infty^\sigma$ is central.
\item $\A_\infty\cup\A_\infty^\sigma$ is an arrangement with maximum number of intersections of multiplicity 3 if and only if $\A^\sigma \in X\in\L_3(B(6,2,\A^\sigma))$ simple non very generic intersection of multiplicity $4$ .
\end{enumerate}
 \end{thm}
\noindent
The above statement only depends on the combinatorics of discriminantal arrangement allowing us to generalize it into the following conjecture. 

\begin{con}\label{con}
 Let $\B(n,3,\mathcal{A}_\infty)$ be a discrimanantal arrangement with maximum number of independent intersections in $\L_2(\B(n,3,\A_\infty))$ and $\sigma\in S$ an element which acts strongly on $\B(n,3,\A_\infty)$, then 
 \begin{enumerate}
 \item the arrangement $\A_\infty\cup\A_\infty^\sigma$ is an arrangement with minimum number of intersections of multiplicity 2 if and only if $\L_2(\B(n,3,\A_\infty))$ also contains the maximum number of purely dependent intersections fixed by $\sigma$ and $\A_\infty^\sigma$ is central;
\item $\A_\infty\cup\A_\infty^\sigma$ is an arrangement with maximum number of intersections of multiplicity 3 if and only if $\A^\sigma \in X\in\L_{n-3}(B(n,2,\A^\sigma))$ is simple non very generic intersection of maximum multiplicity.
\end{enumerate}
\end{con}

\bigskip 

\textbf{Data availability statement:} Not applicable

\end{document}